\documentclass[12pt]{amsart}

\usepackage{amssymb,latexsym}

\usepackage{enumerate}

\makeatletter

\@namedef{subjclassname@2010}{

  \textup{2010} Mathematics Subject Classification}

\makeatother
\newtheorem{thm}{Theorem}[section]

\newtheorem{lem}[thm]{Lemma}

\theoremstyle{definition}

\numberwithin{equation}{section}

\newcommand{\X}{\mathbb{X}}
\newcommand{\Y}{\mathbb{Y}}
\newcommand{\ex}{\mathbb{E}}
\newcommand{\re}{\textup{Re}}

\newcommand{\F}{\mathcal{F}}
\newcommand{\D}{\mathcal{D}}
\newcommand{\sums}{\sideset{}{^\flat}\sum}

\begin{document}

\baselineskip=17pt

\title[The number of imaginary quadratic fields with a given class number]
{On the average of the number of imaginary quadratic fields with a given class number}

\author[Youness Lamzouri]{Youness Lamzouri}

\address{Department of Mathematics and Statistics,
York University,
4700 Keele Street,
Toronto, ON,
M3J1P3
Canada}

\email{lamzouri@mathstat.yorku.ca}

\date{}

\begin{abstract}  
Let $\F(h)$ be the number of imaginary quadratic fields with class number $h$. In this note,  we improve the error term in Soundararajan's asymptotic formula for the average of $\F(h)$. Our argument leads to a similar refinement of the asymptotic for the average of $\F(h)$ over odd $h$, which was recently obtained by Holmin, Jones,  Kurlberg, McLeman and  Petersen. 
\end{abstract}


\thanks{ The author is partially supported by a Discovery Grant from the Natural Sciences and Engineering Research Council of Canada.}

\maketitle

\section{Introduction}

An important problem in number theory, which goes back to Gauss,  is to determine all imaginary quadratic fields with a given class number. Let $\F(h)$ be the number of imaginary quadratic fields with class number $h$. Then for instance one has $\F(1)=9$, which follows from the celebrated 
solution of Baker-Stark-Heegner to Gauss' class
number $1$ problem for imaginary quadratic fields. In \cite{So}, Soundararajan studied the quantity $\F(h)$ and determined its average order. More precisely, he proved that for any $ \epsilon>0$
\begin{equation}\label{SoundAsymp} 
\sum_{h\leq H} \F(h)= \frac{3\zeta(2)}{\zeta(3)} H^2 + O_{\epsilon}\left(\frac{H^2}{(\log H)^{1/2-\epsilon}}\right). 
\end{equation}
The purpose of this note is to improve the error term in this asymptotic formula. 
\begin{thm}\label{Sound} We have
$$\sum_{h\leq H} \F(h)= \frac{3\zeta(2)}{\zeta(3)} H^2 + O\left(\frac{H^2(\log \log H)^3}{\log H}\right). $$

\end{thm}

In a recent work  \cite{HJKMP},  Holmin, Jones, Kurlberg, Mcleman and Petersen studied statistics of the class numbers of imaginary quadratic fields. In particular, they used the Cohen-Lenstra heuristics together with the work of Granville and Soundararajan \cite{GrSo} on the distribution of values of $L(1,\chi_d)$ to formulate a conjecture on the asymptotic nature of $\F(h)$ as $h\to \infty$ through odd values.  They also obtained the analogue of \eqref{SoundAsymp} for the average of $\F(h)$ over odd values of $h$, conditionally on the generalized Riemann hypothesis GRH. More precisely, they showed that assuming GRH
\begin{equation}\label{HJKMPAsymp}
\sum_{\substack{h\leq H\\ h \text{ odd }}} \F(h)= \frac{15}{4} \frac{H^2}{\log H} +  O_{\epsilon}\left(\frac{H^2}{(\log H)^{3/2-\epsilon}}\right).
\end{equation}
Unlike \eqref{SoundAsymp} which is unconditional, the proof of \eqref{HJKMPAsymp} uses GRH to bound a certain character sum over primes,  which appears in this case due to the fact that when $d>8$,  the class number of $\mathbb{Q}(\sqrt{-d})$ is odd precisely when $d$ is prime, by genus theory. 

The same argument in our proof of Theorem \ref{Sound} leads to the following refinement of the asymptotic formula \eqref{HJKMPAsymp}. 
 \begin{thm}\label{HJKMP} Assume GRH. Then
 $$\sum_{\substack{h\leq H\\ h \text{ odd }}} \F(h)= \frac{15}{4} \frac{H^2}{\log H} + 
 O\left(\frac{H^2(\log \log H)^3}{(\log H)^2}\right). $$
 \end{thm}
The main ingredients in the proofs of \eqref{SoundAsymp} and \eqref{HJKMPAsymp} are asymptotic formulas for the complex moments of the class number $h(d)$ in a large uniform range (see \eqref{Moments} and \eqref{Moments2} below). Using this approach, the best saving one can hope for in the error terms of Theorems \ref{Sound} and \ref{HJKMP} will be $1/L$, if one can obtain an asymptotic formula for the average of $h(d)^s$, uniformly in $s$ such that $|s|\leq L$. It is also known (see \cite{GrSo}) that the current methods for computing these moments fail when $L\geq (\log H)(\log\log H)^2$. This shows that the saving of $(\log H)/(\log\log H)^3$ in the error terms of Theorems \ref{Sound} and \ref{HJKMP} constitute (up to the power of $\log\log H$) the limit of Soundararajan's method  \cite{So}.  In particular, it would be interesting to improve the power of $\log H$ in the error terms of these results.  


\section{Proofs of Theorems \ref{Sound} and \ref{HJKMP}}

 Let $X:= H^2 \log\log H$. As in \cite{So}, it follows from Theorem 4 of \cite{GrSo} (concerning the distribution of extreme values of $L(1,\chi_d)$) together with Tatuzawa's refinement of the Landau-Siegel Theorem \cite{Ta}  that 
\begin{equation}\label{Main}
 \sum_{h\leq H} \F(h)= \sums_{\substack{d\leq X\\ h(-d)\leq H}}1 +O_A\left(\frac{H^2}{(\log H)^A}\right),
 \end{equation}
for any $A>0$, where $\flat$ indicates that the sum is over fundamental discriminants $-d$.

 To estimate the main term in \eqref{Main}, Soundararajan used the following variant of Perron's formula $$ \frac{1}{2\pi i} \int_{c-i\infty}^{c+i\infty} \frac{x^s}{s} \left(\frac{(1+\delta)^{s+1}-1}{\delta(s+1)}\right)ds=\begin{cases} 1 & \text{ if } x\geq 1,\\
(1+\delta-1/x)/\delta & \text{ if } (1+\delta)^{-1}\leq x\leq 1, \\ 0 & \text{ if } 0< x\leq (1+\delta)^{-1}.
\end{cases}$$ 
Our improvement comes from using a different smooth cut-off function, namely
$$
I_{c, \lambda, N}(y):=\frac{1}{2\pi i}\int_{c-i\infty}^{c+i\infty} y^s \left(\frac{e^{\lambda s}-1}{\lambda s}\right)^N\frac{ds}{s},
$$
where $c,\lambda>0$ are real numbers and $N$ is a positive integer. We prove
\begin{lem}\label{SmoothPerron}
Let $\lambda, c >0$ be real numbers and $N$ be a positive integer. Then we have 
$$
I_{c, \lambda, N}(y)
\begin{cases}=1 & \text{ if } y>1, \\ \in [0,1] & \text{ if } e^{-\lambda N } \leq y \leq 1,\\
 =0 & \text{ if } 0<y< e^{-\lambda N }. \end{cases}
$$
\end{lem}
\begin{proof}
First, we recall Perron's formula
\begin{equation}\label{Perron}
\frac{1}{2\pi i}\int_{c-i\infty}^{c+i\infty} y^s \frac{ds}{s} = \begin{cases}1 & \text{ if } y> 1, \\ \frac{1}{2} & \text{ if } y = 1,\\
0 & \text{ if } 0<y<1. \end{cases}
\end{equation}
Then, we observe that
$$\frac{1}{2\pi i}\int_{c-i\infty}^{c+i\infty} y^s \left(\frac{e^{\lambda s}-1}{\lambda s}\right)^N\frac{ds}{s}
= \frac{1}{\lambda^N}\int_{0}^{\lambda}\cdots \int_0^{\lambda} \frac{1}{2\pi i} \int_{c-i\infty}^{c+i\infty}
\left(ye^{t_1+ \cdots+ t_N}\right)^s\frac{ds}{s} dt_1\cdots dt_N.
$$
By \eqref{Perron}, $\frac{1}{2\pi i} \int_{c-i\infty}^{c+i\infty}
\left(ye^{t_1+ \cdots+ t_N}\right)^s\frac{ds}{s}\in  [0, 1]$ for all values of $t_i$, and hence $I_{c, \lambda, N}(y) \in [0, 1]$ for all $y>0$. The lemma follows from \eqref{Perron} upon noting that $ye^{t_1+ \cdots+ t_N}>1$ for all $t_i\in [0, \lambda]$ if $y>1$, and $ye^{t_1+ \cdots+ t_N}<1$ for all $t_i\in [0, \lambda]$ if  $ 0<y< e^{-\lambda N}$. 
\end{proof}

\begin{proof}[Proof of Theorem \ref{Sound}]
Let $c=1/\log H$, $N$ be a positive integer, and $0<\lambda \leq 1$ be a real number to be chosen later. By \eqref{Main} and Lemma \ref{SmoothPerron} we obtain 
\begin{equation}\label{Bounds}
\sum_{h\leq H} \F(h)
\leq \frac{1}{2\pi i}\int_{c-i\infty}^{c+i\infty}\sums_{d\leq X}\frac{H^s}{h(-d)^s}\left(\frac{e^{\lambda s}-1}{\lambda s}\right)^N\frac{ds}{s} +O_A\left(\frac{H^2}{(\log H)^A}\right)\leq \sum_{h\leq e^{\lambda N}H} \F(h).
\end{equation}

Let $T:= \log X/(10^4(\log\log X)^2).$ Then, it follows from equation (5) of \cite{So} that
\begin{equation}\label{Moments}
 \sums_{d\leq X} h(-d)^{-s}= 3\pi^{s-2}\cdot \ex\left(L(1,\X)^{-s}\right)\int_1^X x^{-s/2}dx+ 
 O\left(X\exp\left(-\frac{\log X}{5\log\log X}\right)\right),
\end{equation}
for all complex numbers $s$ with $\re(s)=c$ and $|s|\leq T$,
where 
\begin{equation}\label{EulerProduct}
L(1, \X)= \prod_{p} \left(1-\frac{\X(p)}{p}\right)^{-1},
\end{equation}
and $\{\X(p)\}$ is a sequence of independent random variables taking the value $1$ with probability $p/(2(p+1))$, $0$ with probability $1/(p+1)$, and $-1$ with probability $p/(2(p+1))$.
Note that $\ex(\X(p))=0$ and $\ex(\X(p)^2)\leq 1$, and hence the random product \eqref{EulerProduct} converges almost surely, by Kolmogorov's three series theorems. 

Since  $|e^{\lambda s}-1|\leq 3$ (if $H$ is large enough) and $h(-d)\geq 1$, it follows that the contribution of the region $|s|>T$ to the integral in \eqref{Bounds} is
$$
\ll X\left(\frac{3}{\lambda}\right)^N \int_{\substack{|s|> T\\ \re(s)=c}} \frac{|ds|}{|s|^{N+1}}\ll \frac{X}{N}\left(\frac{3}{\lambda T}\right)^N.
$$
Moreover, note that $|(e^{\lambda s}-1)/\lambda s|\leq 3$, if $H$ is large enough. Therefore, if follows from \eqref{Moments} that the integral in \eqref{Bounds} equals
\begin{equation}\label{MTerm}
\frac{1}{2\pi i}\int_{\substack{|s|\leq T\\ \re(s)=c}} \frac{3}{\pi^2}\cdot \ex\left(L(1,\X)^{-s}\right)\left(\int_1^X x^{-s/2}dx\right)
 (\pi H)^s \left(\frac{e^{\lambda s}-1}{\lambda s}\right)^N\frac{ds}{s} + \mathcal{E},\end{equation}
 where 
 $$\mathcal{E}\ll \frac{X}{N}\left(\frac{3}{\lambda T}\right)^N+ \frac{3^N T }{c} X\exp\left(-\frac{\log X}{5\log\log X}\right).$$
 Choosing $\lambda = 10/T$ and $N=[A\log\log H]$, where $A>1$ is a constant, implies that\begin{equation}\label{ErrorCut}
\mathcal{E}\ll_A \frac{H^2}{(\log H)^A}.
\end{equation}
Furthermore, extending  the main term of \eqref{MTerm} to $\int_{c-i\infty}^{c+i\infty}$ shows that this integral equals
\begin{equation}\label{MTerm2}
\begin{aligned}
 &\frac{1}{2\pi i}\int_{c-i\infty}^{c+i\infty} \frac{3}{\pi^2}\cdot \ex\left(L(1,\X)^{-s}\right)\left(\int_1^X x^{-s/2}dx\right)
 (\pi H)^s \left(\frac{e^{\lambda s}-1}{\lambda s}\right)^N\frac{ds}{s} \\
 & \quad  \quad  + O_A\left( \ex\left(L(1,\X)^{-c}\right)\frac{X}{N}\left(\frac{3}{\lambda T}\right)^N\right) \\
 =& \frac{3}{\pi^2}\cdot \ex\left(\int_1^X I_{c, \lambda, N} \left(\frac{\pi H }{\sqrt{x}}L(1, \X)^{-1}\right)dx\right)+ 
 O_A\left(\frac{H^2}{(\log H)^A}\right).
 \end{aligned}
 \end{equation}
 
 Now, it follows from Lemma \ref{SmoothPerron} that for any $1\leq x\leq X$ we have
$$ I_{c, \lambda, N} \left(\frac{\pi H }{\sqrt{x}}L(1, \X)^{-1}\right) \begin{cases} = 1 & \text{ if } \sqrt{x}L(1, \X)\leq \pi H , \\ \in [0,1] & \text{ if } \pi H< \sqrt{x}L(1, \X)\leq e^{\lambda N }\pi H,\\
= 0 & \text{ if } \sqrt{x} L(1, \X)> \pi H e^{\lambda N }. \end{cases}$$
Thus we obtain
\begin{equation}\label{Expectation}
\begin{aligned}
 \ex\left(\int_1^X I_{c, \lambda, N} \left(\frac{\pi H }{\sqrt{x}}L(1, \X)^{-1}\right)dx\right) 
&=\ex\left(\min\left( \frac{\pi^2 H^2}{L(1,\X)^2}, X\right) +O\left(\frac{H^2 (e^{2\lambda N}-1)}{L(1,\X)^2}\right)\right)\\
&= \ex\left(\min\left( \frac{\pi^2 H^2}{L(1,\X)^2}, X\right)\right)+ O\left(\frac{H^2(\log\log H)^3}{\log H}\right).
\end{aligned}
\end{equation}

Finally, using Proposition 1 of \cite{GrSo} (which states that the probability that $L(1,\X)\leq \pi^2/(6e^{\gamma}\tau)$ is $\exp(- e^{\tau-C_1}/\tau (1+o(1))$ for some explicit constant $C_1$), we obtain
\begin{align*}
 \ex\left(\min\left( \frac{\pi^2 H^2}{L(1,\X)^2}, X\right)\right)
 &= \pi^2 H^2 \cdot \ex\left(L(1,\X)^{-2}\right)+ O_A\left(\frac{H^2}{(\log H)^A}\right)\\
 &= \frac{\pi^2\zeta(2)}{\zeta(3)}H^2+ O_A\left(\frac{H^2}{(\log H)^A}\right).\\
\end{align*}
 Combining this estimate with equations \eqref{Bounds}, \eqref{MTerm}--\eqref{Expectation}, and noting that $(He^{\lambda N})^2-H^2 \ll H^2(\log\log H)^3/\log H$ completes the proof. 
 \end{proof}
 
 \begin{proof}[Proof of Theorem \ref{HJKMP}]
 Let $X$, $T$ and $c$ be as in the proof of Theorem \ref{Sound}. Let $\D(x)= \{ p\leq x : p\equiv 3\pmod 4\}$. Then similarly to \eqref{Main}, one has (see equation 3.6 of \cite{HJKMP})
$$
 \sum_{\substack{h\leq H \\ h \text{ odd }}} \F(h)= \sum_{\substack{p \in  \D(X)\\ h(-p)\leq H}}1 +O_A\left(\frac{H^2}{(\log H)^A}\right).
$$

 Let $\{\Y(p)\}$ be independent random variables taking the values $1$ and $-1$ with equal probabilities $1/2$, and define 
 $$
 L(1,\Y)=\prod_{p}\left(1-\frac{\Y(p)}{p}\right)^{-1}.
 $$
 
 To obtain \eqref{HJKMPAsymp},  the authors of \cite{HJKMP} prove that assuming GRH (see Theorem 3.3 of \cite{HJKMP}) we have
 $$ \sum_{p\in \D(x)} L(1,\chi_p)^z=  |\D(x)| \cdot \ex\left(L(1, \Y)^z\right) +O_{\epsilon}\left(x^{1/2+\epsilon}\right),$$
uniformly for all complex numbers $z$ such that $|z|\leq (\log x)/(500 (\log\log x)^2)$, where $L(s, \chi_p)$ is the Dirichlet $L$-function attached to the Kronecker symbol $\chi_p=\left(\frac{-p}{\cdot}\right)$. Then, by partial summation together with Dirichlet's class number formula, they deduced that (see \cite{HJKMP}, p. 19)
\begin{equation}\label{Moments2} \sum_{p\in \D(X)} h(-p)^{-s}= \pi^s  \cdot \ex\left(L(1, \Y)^s\right) \int_1^X x^{-s/2}d |\D(x)|+ 
 O_{\epsilon}\left(X^{1/2+\epsilon}\right),
\end{equation}
for all complex numbers $s$ with $\re(s)=c$ and $|s|\leq T$.

 The proof of Theorem \ref{HJKMP} then follows along the same lines of the proof of Theorem \ref{Sound} by using \eqref{Moments2} instead of \eqref{Moments}.
 
 \end{proof}


\begin{thebibliography}{DDDD}



\bibitem[1] {GrSo} A. Granville and K. Soundararajan,
\emph{The distribution of values of $L(1, \chi_d)$.} 
Geom. Funct. Anal. 13 (2003), no. 5, 992--1028. 

\bibitem[2] {HJKMP} S. Holmin, N. Jones, P. Kurlberg, C. McLeman, K. L. Petersen, 
\emph{Missing class groups and class number statistics for imaginary quadratic fields.} 
Preprint, 28 pages. 	arXiv:1510.04387.

\bibitem[3]{So}  K. Soundararajan,
\emph{The number of imaginary quadratic fields with a given class number.} 
Hardy-Ramanujan J. 30 (2007), 13--18.

\bibitem[4] {Ta} T. Tatuzawa,
\emph{On a theorem of Siegel. }
Jap. J. Math. 21 (1951), 163--178. 



\end{thebibliography}
\end{document}